\documentclass[11pt]{amsart}

\usepackage[margin=1.25in]{geometry}

\usepackage{graphicx}

\usepackage{wasysym}
\usepackage[misc]{ifsym}

\usepackage{amsmath, amssymb, amsthm, graphicx, enumerate, tikz, float, bbm, nicefrac, mathrsfs}
\usepackage{mathtools, comment}
\usepackage{subcaption}
\usepackage[colorlinks]{hyperref}
\hypersetup{citecolor=blue}
\usetikzlibrary{matrix,arrows,decorations.pathmorphing}

\newtheorem{theorem}{Theorem}[section]
\newtheorem{lemma}[theorem]{Lemma}
\newtheorem{corollary}[theorem]{Corollary}
\newtheorem{proposition}[theorem]{Proposition}

\newcommand{\n}[1]{||#1||}

\theoremstyle{definition}
\newtheorem{definition}[theorem]{Definition}

\newtheorem*{solution*}{Solution}

\theoremstyle{remark}
\newtheorem{remark}[theorem]{Remark}
\newtheorem*{pf*}{Pf}
\newtheorem*{pfbase*}{Proof of Base Case}
\newtheorem*{pfstep*}{Proof of Inductive Step}

\numberwithin{equation}{section}

\newcommand{\C}{\mathbb{C}}
\newcommand{\R}{\mathbb{R}}
\newcommand{\Z}{\mathbb{Z}}
\newcommand{\U}{\mathscr{U}}
\newcommand{\M}{\mathscr{M}}

\begin{document}

\allowdisplaybreaks

\title[Invariant Spaces of Measurable Functions]{Certain Invariant Spaces of Bounded Measurable Functions on a Sphere}

\author{Samuel A. Hokamp}
\address{Department of Mathematics, St. Norbert College, De Pere, Wisconsin 54115}
\email{samuel.hokamp@snc.edu}

\subjclass[2010]{Primary 46E30. Secondary 32A70}

\date{\today}

\keywords{Spaces of measurable functions, several complex variables, functional analysis.\\\indent\emph{Corresponding author.} Samuel A. Hokamp \Letter~\href{mailto:samuel.hokamp@snc.edu}{samuel.hokamp@snc.edu} \phone~920-403-2947.}

\begin{abstract}
In their 1976 paper, Nagel and Rudin characterize the closed unitarily and M\"obius invariant spaces of continuous and $L^p$-functions on a sphere, for $1\leq p<\infty$. In this paper we provide an analogous characterization for the weak*-closed unitarily and M\"obius invariant spaces of $L^\infty$-functions on a sphere. We also investigate the weak*-closed unitarily and M\"obius invariant algebras of $L^\infty$-functions on a sphere.
\end{abstract}

\maketitle

\section{Introduction}

Let $n$ be a positive integer and $S$ the unit sphere of $\C^n$. We say a space of complex functions defined on $S$ is \textit{unitarily invariant} if the pre-composition of any function in the set with a unitary transformation remains in the set. Similarly, we say a space of complex functions defined on $S$ is \textit{M\"obius invariant} if the pre-composition of any function in the set with a biholomorphic map of the complex unit ball onto itself remains in the set. It is clear that M\"obius invariance implies unitary invariance.

In \cite{NR}, Nagel and Rudin determine the closed unitarily and M\"obius invariant spaces of continuous and $L^p$-functions on $S$, for $1\leq p<\infty$. The closed M\"obius invariant algebras of continuous and $L^p$-functions on $S$, for $1\leq p<\infty$, are also characterized in \cite{NR}, and the closed unitarily invariant algebras of continuous functions on $S$ are considered in \cite{RUA}.

In this paper, we consider $L^\infty(S)$ with the weak*-topology, thereby obtaining results for $L^\infty(S)$ which are analogous to Nagel and Rudin's results for continuous functions.

In Section \ref{U inv spaces}, we determine the weak*-closed unitarily invariant subspaces of $L^\infty(S)$ (Theorem \ref{weak*-closed U inv}), and in Section \ref{U inv algebras}, we show that those weak*-closed unitarily invariant subspaces of $L^\infty(S)$ which are algebras correspond in a certain way to the closed unitarily invariant subalgebras of continuous functions (Theorem \ref{algebra patterns}). This correspondence leads to several results for weak*-closed unitarily invariant subalgebras of $L^\infty(S)$ (Theorem \ref{Linf G(d)}, Theorem \ref{Linf G(sig)}, and Theorem \ref{smallest Linf alg pattern}) that are analogous to results for closed unitarily invariant algebras of continuous functions from \cite{RUA}. These results are given in Section \ref{applications}.

In Section \ref{M inv spaces}, we determine the weak*-closed M\"obius invariant subspaces of $L^\infty(S)$ (Theorem \ref{M-inv subspaces of Linf}), and in Section \ref{M inv algebras}, we determine which of the weak*-closed M\"obius invariant subspaces of $L^\infty(S)$ are algebras (Theorem \ref{M-inv subalgebras of Linf}).

\section{Preliminaries} 

\subsection{Notation and Definitions}

We let $n$ be a positive integer and $S$ and $B$ be the unit sphere and unit ball of $\C^n$, respectively, and $\overline{B}$ the closed unit ball.

We let $\sigma$ denote the unique \textit{rotation-invariant} positive Borel measure on $S$ for which $\sigma(S)=1$. The phrase ``rotation-invariant'' refers to the orthogonal group $O(2n)$, the group of isometries on $\R^{2n}$ that fix the origin. The notation $L^p(S)$ then denotes the usual Lebesgue spaces, for $1\leq p\leq\infty$, in reference to the measure $\sigma$, with the usual norms $\n{\cdot}_p$.

For $Y\subset C(S)$, the uniform closure of $Y$ is denoted $\overline{Y}$, and for $Y\subset L^p(S)$, $1\leq p\leq\infty$, the norm-closure of $Y$ in $L^p(S)$ is denoted $\overline{Y}^{p}$. Additionally, for $Y\subset L^\infty(S)$, the weak*-closure of $Y$ in $L^\infty(S)$ is denoted $\overline{Y}^{*}$.

Occasionally, we let $X$ denote either $C(S)$ or one of the spaces $L^p(S)$ for $1\leq p\leq\infty$. Then, if $Y\subset X$, the norm-closure of $Y$ in $X$ is denoted $\overline{Y}^X$.

\begin{remark}\label{weak* and weak Lp closure containment}
For $Y\subset L^\infty(S)$ convex and $1\leq p<\infty$, we have $$\overline{Y}^{*}\subset \overline{Y}^{p}\cap L^\infty(S).$$
\end{remark}

This follows from the fact that the topology $L^\infty(S)$ inherits from the weak topology on each $L^p(S)$ is weaker than the weak*-topology, and from the local convexity of $L^p(S)$.

We let $\U=\U(n)$ denote the group of unitary operators on the Hilbert space $\C^n$. Then $\U$ is a compact subgroup of $O(2n)$, and we use $dU$ to denote the Haar measure on $\U$.

\begin{proposition}[\cite{RFT} Proposition 1.4.7]\label{prop147}
Let $f\in L^1(S)$ and $z\in S$. Then $$\int_S f\,d\sigma=\int_\U f(Uz)\,dU.$$
\end{proposition}

\begin{remark}\label{int U-invariant}
Let $f\in L^1(S)$ and $U\in\U$. Then $$\int_S f\,d\sigma=\int_S f\circ U\,d\sigma.$$
\end{remark}

We let $\M=\M(n)$ denote the \textit{M\"obius group} in dimension $n$. This is the group of injective holomorphic maps of $B$ onto $B$. Important to note is that each element of $\M$ extends to a homeomorphism of $\overline{B}$ onto $\overline{B}$, and thus maps $S$ onto $S$.

We let $Q$ denote the set of pairs of nonnegative integers; that is, $$Q=\{(p,q):p,q\in\Z\text{ and }p\geq 0,q\geq 0\}.$$

On the sets $S$ and $B$ we define the following spaces of complex functions:
\begin{itemize}
    \item[(1)] $C(S)$ is the space of continuous functions on $S$.
    \item[(2)] $A(S)$ is the space of functions which are restrictions to $S$ of holomorphic functions on $B$ that are continuous on $\overline{B}$.
    \item[(3)] $H^\infty(B)$ is the space of bounded holomorphic functions on $B$.
    \item[(4)] $H^\infty(S)$ is the space of functions in $L^\infty(S)$ which are radial limits almost everywhere of functions in $H^\infty(B)$.
\end{itemize}

If $X$ is a space of complex functions, then $\text{Conj}(X)$ denotes the space of conjugates of members of $X$.

\begin{definition}[\cite{NR} Section IV.]
We say a space of complex functions $Y$ with domain $S$ is \textbf{unitarily invariant} ($\U$-invariant) if $f\circ U\in Y$ for every $f\in Y$ and every $U\in\U$.
\end{definition}

\begin{definition}[\cite{NR} Section I.]
We say a space of complex functions $Y$ with domain $S$ is \textbf{M\"obius invariant} ($\M$-invariant) if $f\circ\phi\in Y$ for every $f\in Y$ and every $\phi\in\M$.
\end{definition}

\subsection{The Spaces \texorpdfstring{$H(p,q)$}{H(p,q)} and the Maps \texorpdfstring{$\pi_{pq}$}{ppq}}

For each $(p,q)\in Q$, we define $H(p,q)$ to be the space of harmonic, homogeneous polynomials on $S$ with total degree $p$ in the variables $z_1,z_2,\cdots,z_n$ and total degree $q$ in the variables $\bar{z}_1,\bar{z}_2,\cdots,\bar{z}_n$. These spaces $H(p,q)$ are finite-dimensional (hence closed) subspaces of $C(S)$ and $L^2(S)$.

If $\Omega$ is any subset of $Q$, we define the set $E_\Omega$ to be the algebraic sum of the spaces $H(p,q)$ such that $(p,q)\in\Omega$. That is, $$E_\Omega=\Big\{\sum_{(p,q)\in A} f_{pq}:f_{pq}\in H(p,q),\text{ }A\text{ any finite subcollection of }\Omega\Big\}.$$

Our notation for the spaces $H(p,q)$ and $E_\Omega$ is the same used by Rudin in \cite{RFT}. For the maps $\pi_{pq}$ (Theorem \ref{L2 Hpq structure}(c)), the notation in \cite{NR} and \cite{RFT} is the same, so we use it as well.

\begin{theorem}[\cite{NR} Theorem 4.6]\label{L2 Hpq structure}
Let $(p,q)\in Q$.
\begin{itemize}
    \item[(a)] For every $z\in S$, $H(p,q)$ contains a unique function $K_z$ that satisfies
    \begin{itemize}
        \item[(i)] $K_z\circ U=K_z$ for every $U\in\U$ with $Uz=z$, and
        \item[(ii)] $\displaystyle\int_S|K_z|^2\,d\sigma=K_z(z)>0$.
    \end{itemize}
    \item[(b)] $H(p,q)$ is $\U$-minimal; that is, $H(p,q)$ is $\U$-invariant with no proper $\U$-invariant subspaces.
    \item[(c)] If $X$ is a function space with domain $S$ and $\pi_{pq}$ is defined for $f\in X$ by $$\pi_{pq}f(z)=\int_S f\overline{K_z}\,d\sigma,$$ then $\pi_{pq}$ is a projection of $X$ onto $H(p,q)$.
    \item[(d)] Each $\pi_{pq}$ is an orthogonal projection in $L^2(S)$. Further, the spaces $H(p,q)$ are pairwise orthogonal in $L^2(S)$, and $L^2(S)$ is the direct sum of the spaces $H(p,q)$. That is, $$L^2(S)=\overline{E}_Q^{2}$$
\end{itemize}
\end{theorem}

Observe that each $E_\Omega$ is $\U$-invariant since each $H(p,q)$ is $\U$-invariant.

\begin{remark}\label{f_pq=pi_pq f}
Explicitly, Theorem \ref{L2 Hpq structure}(d) says that each $f\in L^2(S)$ has a unique expansion $f=\sum f_{pq}$, with each $f_{pq}\in H(p,q)$, which converges unconditionally to $f$ in the $L^2$-norm. Since $\pi_{pq}$ is the identity map on $H(p,q)$ and the spaces $H(p,q)$ are pairwise orthogonal, we have $f_{pq}=\pi_{pq}f$ for $(p,q)\in Q$. Thus, $$f=\sum\pi_{pq}f.$$
\end{remark}

Each $\pi_{pq}$ is continuous as the orthogonal projection of $L^2(S)$ onto the finite-dimensional subspace $H(p,q)$. Thus, $\pi_{pq}$ annihilates a subset of $L^2(S)$ if and only if it annihilates its closure. This proves the following theorem:

\begin{theorem}\label{pi pq closure}
If $Y\subset L^2(S)$, then $\pi_{pq}Y=0$ if and only if $\pi_{pq}\overline{Y}^{2}=0$.
\end{theorem}

\begin{remark}\label{E Omega L2 pi pq}
For each set $\Omega\subset Q$, we have $$\overline{E}_\Omega^{2}=\{f\in L^2(S):\pi_{pq}f=0\text{ when }(p,q)\notin\Omega\}.$$
This follows from Remark \ref{f_pq=pi_pq f} and Theorem \ref{pi pq closure}.
\end{remark}

\subsection{\texorpdfstring{$\U$}{U}- and \texorpdfstring{$\M$}{M}-Invariant Subspaces and Subalgebras of \texorpdfstring{$C(S)$}{C(S)}}

In this section we summarize the results of Nagel and Rudin concerning the unitarily and M\"obius invariant subspaces and subalgebras of $C(S)$. The results we draw on for this paper are formulated in \cite{NR} and \cite{RUA}, but the reader can also find them presented in \cite{RFT}. 

In \cite{NR}, Nagel and Rudin establish that each closed $\U$-invariant subspace of $C(S)$ or $L^p(S)$ for $1\leq p<\infty$ corresponds to some $\Omega\subset Q$. This $\Omega$ is the set of all $(p,q)\in Q$ such that $\pi_{pq}$ does not annihilate the space.

\begin{theorem}[\cite{NR} Theorem 4.4]\label{main C and Lp result}
Let $X$ be any of the spaces $C(S)$ or $L^p(S)$ for $1\leq p<\infty$. If $Y$ is a closed $\U$-invariant subspace of $X$, then $Y=\overline{E}_\Omega^X$ for some $\Omega\subset Q$.
\end{theorem}

The proof of Theorem \ref{main C and Lp result} in \cite{RFT} is different than in \cite{NR}. In particular, the proof in \cite{RFT} uses Lemma \ref{main C(S) lemma}, whose proof in turn requires Lemma \ref{U to X cont}.

\begin{lemma}[\cite{RFT} Lemma 12.3.5]\label{main C(S) lemma}
If $Y\subset C(S)$ is a closed $\U$-invariant space. Then for $g\in C(S)$, we have that $g\notin\overline{Y}^{2}$ whenever $g\notin\overline{Y}$.
\end{lemma}

\begin{lemma}[\cite{RFT} Lemma 12.3.3]\label{U to X cont}
Let $X$ be any of the spaces $C(S)$ or $L^p(S)$ for $1\leq p<\infty$. If $f\in X$, then $U\mapsto f\circ U$ is a continuous map of $\U$ into $X$.
\end{lemma}

The characterization of the closed $\U$-invariant subalgebras of $C(S)$ is dependent on the dimension $n$. The case when $n=1$ is simple and is covered in \cite{NR}: Let $\Omega$ be any additive semigroup of integers. The set of continuous functions on $S$ whose Fourier coefficients vanish on the complement of $\Omega$ is a closed $\U$-invariant algebra, and there are no others.

In \cite{RUA}, Rudin develops a combinatorial criterion that describes the sets $\Omega\subset Q$ (called \textbf{algebra patterns}) which induce closed $\U$-invariant subalgebras of $C(S)$ when $n\geq 3$.

\begin{theorem}[\cite{RUA} Theorem 1]\label{criterion}
The following property of a set $\Omega\subset Q$ implies $\Omega$ is an algebra pattern: If $(p,q),(r,s)\in\Omega$ and $$\mu=\min\{p+q,r+s,p+r,q+s\},$$ then $\Omega$ contains all points $(p+r+j,q+s+j)$ with $0\leq j\leq\mu$. Conversely, every algebra pattern has this property when $n\geq 3$.
\end{theorem}

The case $n=2$ has some exceptions which Rudin also addresses in \cite{RUA}. Theorem \ref{criterion} can also be formulated in terms of the spaces in the following definition.

\begin{definition}[\cite{RUA} Definition 1.5]\label{H(p,q)H(r,s)}
Given $(p,q),(r,s)\in Q$, we define $H(p,q)\cdot H(r,s)$ to be the linear span of products $fg$, where $f\in H(p,q)$ and $g\in H(r,s)$.
\end{definition}

Theorem \ref{criterion} then yields several examples of algebra patterns, namely the sets described in Definition \ref{Rudin alg patts}. As in \cite{RUA}, we define $\Delta_k$ to be the set $\{(p,q)\in Q:p-q=k\}$ for each integer $k$.

\begin{definition}[\cite{RUA} Definition 1.6]\label{Rudin alg patts}
Let $d$ be any integer, $\Sigma$ any (possibly empty) additive semigroup of positive integers, and $(p,q)\in Q$ such that $0\leq q<p$. We define
\begin{itemize}
    \item[(1)] $\displaystyle G(d)$ to be the union of all $\Delta_{kd}$, for $-\infty<k<\infty$,
    \item[(2)] $\displaystyle G(\Sigma)$ to be the union of $\Delta_0$ and all $\Delta_k$ for $k\in\Sigma$, and
    \item[(3)] $G(p,q)$ to be the set consisting of $(p,q)$ and points $$(mp-j,mq-j),$$ where $m=2,3,4,\ldots$, and $0\leq j\leq mq$.
\end{itemize}
\end{definition}

The following theorems of \cite{RUA} show the importance of these algebra patterns. The first two results are also in \cite{NR} and describe every self-adjoint closed $\U$-invariant algebra of continuous functions on $S$ (with $\Sigma$ empty).

\begin{theorem}[\cite{RUA} Theorem 2]\label{RUA Thm 2}
If $\Omega$ is an algebra pattern that contains some $(p,q)$ with $p>q$ and some $(r,s)$ with $r<s$, then $\Omega=G(d)$ for some $d>0$.
\end{theorem}

\begin{theorem}[\cite{RUA} Theorem 3]\label{RUA Thm 3}
Suppose $\Omega$ is an algebra pattern such that $p\geq q$ for every $(p,q)\in\Omega$ and $(a,a)\in\Omega$ for some $a>0$.
\begin{itemize}
    \item[(1)] If $n\geq 3$, then $\Omega=G(\Sigma)$.
    \item[(2)] If $n=2$ and if $G^*(\Sigma)$ is obtained from $G(\Sigma)$ by deleting all $(m,m)$ with $m$ odd, then $G^*(\Sigma)$ is also an algebra pattern, and $\Omega$ is either $G(\Sigma)$ or $G^*(\Sigma)$.
\end{itemize}
\end{theorem}

\begin{theorem}[\cite{RUA} Theorem 4]\label{RUA Thm 4}
Fix $(p,q)\in Q$ with $p>q$. Let $\Omega(p,q)$ be the smallest algebra pattern that contains $(p,q)$.
\begin{itemize}
    \item[(1)] If $n\geq 3$, then $\Omega(p,q)=G(p,q)$.
    \item[(2)] If $n=2$, then $\Omega(p,q)$ is obtained from $G(p,q)$ by deleting all $(mp-1,mq-1)$ for $m=2,3,4,\ldots$, and all $(2p-j,2q-j)$ for $j$ odd.
\end{itemize}
\end{theorem}

Regarding the closed $\M$-invariant subspaces of $C(S)$, Nagel and Rudin show in \cite{NR} (Theorem B) that only six such spaces exist:

\begin{itemize}
\item[1.] the null space $\{0\}$,
\item[2.] the space of constant functions $\C$,
\item[3.] the space $A(S)$,
\item[4.] the space $\text{Conj}(A(S))$,
\item[5.] the closure of the space $A(S)+\text{Conj}(A(S))$, and
\item[6.] the space $C(S)$.
\end{itemize}

This is clear from Theorem \ref{main C and Lp result} and the following lemma.

\begin{lemma}[\cite{NR} Lemma 5.2]\label{M C(S) lemma}
Suppose $Y$ is a closed $\M$-invariant subspace of $C(S)$, $p\geq 1$ an integer, and $H(p,q)\subset Y$ for some $q\geq 0$. Then $H(p-1,q)\subset Y$ and $H(p+1,q)\subset Y$.
\end{lemma}

The closed $\M$-invariant subalgebras of $C(S)$ are those spaces above that are also algebras. Observe that in general all are algebras except for the closure of $A(S)+\text{Conj}(A(S))$. This is an algebra only when $n=1$.

\section{Closures of \texorpdfstring{$\U$}{U}-Invariant Sets}\label{Uinv Closures}

In this section we verify that $\U$-invariance is preserved by closures in the spaces $C(S)$ and $L^p(S)$ (Corollaries \ref{Lp and C U-inv closure} and \ref{Linf U-inv closure}). We prove this by showing that $\U$ induces a class of $L^p$-isometries on $L^p(S)$ and a class of isometries on $C(S)$ (Theorem \ref{Phi U norm continuous}), as well as a class of weak*-homeomorphisms on $L^\infty(S)$ (Theorem \ref{weak* cont}).

The following is a consequence of Remark \ref{int U-invariant}. Observe the case when $p=\infty$ follows from the fact that $\sigma$ is rotation-invariant.

\begin{remark}\label{Lp U-invariant}
The space $L^p(S)$ is $\U$-invariant, for $1\leq p\leq\infty$.
\end{remark}

\begin{theorem}\label{Phi U norm continuous}
Suppose $X$ is any of the spaces $C(S)$ or $L^p(S)$ for $1\leq p\leq\infty$ and $U\in\U$. If $\Phi_U:X\to X$ is the map given by $\Phi_U(f)=f\circ U$, then $\Phi_U$ is a bijective linear isometry.
\end{theorem}

\begin{proof}
We begin with the case that $X$ is any of the spaces $L^p(S)$ for $1\leq p\leq\infty$. Remark~\ref{int U-invariant} shows that the map $\Phi_U$ is an $L^p$-isometry (hence injective), and the linearity of $\Phi_U$ is clear.

We need only show $\Phi_U$ is surjective. Observe if $f\in L^p(S)$, then $f\circ U^{-1}\in L^p(S)$ by Remark \ref{Lp U-invariant}. Thus, $$\Phi_U(f\circ U^{-1})=(f\circ U^{-1})\circ U=f,$$ so that $\Phi_U$ is surjective.

The case when $X$ is the space $C(S)$ follows from the previous case when $p=\infty$, since the $L^\infty$-norm coincides with the uniform norm on $C(S)$.
\end{proof}

\begin{corollary}\label{Lp and C U-inv closure}
Suppose $X$ is any of the spaces $C(S)$ or $L^p(S)$ for $1\leq p\leq\infty$. If $Y\subset X$ is $\U$-invariant, then $\overline{Y}^X$ is $\U$-invariant.
\end{corollary}

\begin{theorem}\label{weak* cont}
Let $U\in\U$. If $\Phi_U:L^\infty(S)\to L^\infty(S)$ is the map given by $\Phi_U(f)=f\circ U$, then $\Phi_U$ is a weak*-homeomorphism.
\end{theorem}

\begin{proof}
We recall that the weak*-topology on $L^\infty(S)$ is a weak topology induced by the maps on $L^\infty(S)$ of the form $$\Lambda_gf=\int_S fg\,d\sigma,$$ for some $g\in L^1(S)$. Thus, $\Phi_U$ is continuous with respect to the weak*-topology if and only if $\Lambda_g\circ\Phi_U$ is continuous for all maps $\Lambda_g$.

Fix $g\in L^1(S)$. We observe that $$(\Lambda_g\circ\Phi_U)(f)=\Lambda_g(f\circ U)=\int_S (f\circ U)\cdot g\,d\sigma,$$ for every $f\in L^\infty(S)$. 

By Remark \ref{Lp U-invariant}, we get that $(f\circ U)\in L^\infty(S)$, and so the function $(f\circ U)\cdot g$ is an element of $L^1(S)$. Application of Remark \ref{int U-invariant} and the fact that the inverse map $U^{-1}$ is also a unitary operator yields
\begin{equation}
\small \int_S (f\circ U)\cdot g\,d\sigma=\int_S \big[(f\circ U)\cdot g\big]\circ U^{-1} \,d\sigma=\int_S (f\circ U\circ U^{-1})\cdot (g\circ U^{-1})\,d\sigma=\int_S f\cdot (g\circ U^{-1})\,d\sigma. \label{eq:1}
\end{equation}

Since $L^1(S)$ is $\U$-invariant, we have that $(g\circ U^{-1})\in L^1(S)$. Thus, we have $$(\Lambda_g\circ\Phi_U)(f)=\int_S f\cdot (g\circ U^{-1})\,d\sigma=\Lambda_{g\circ U^{-1}}f,$$ for every $f\in L^\infty(S)$.

Since $\Lambda_g\circ\Phi_U$ is a point evaluation on $L^\infty(S)$ and hence weak*-continuous, we conclude that $\Phi_U$ is continuous on $L^\infty(S)$ with respect to the weak*-topology.

Finally, the map $\Phi_{U^{-1}}:L^\infty(S)\to L^\infty(S)$ given by $\Phi_{U^{-1}}(f)=f\circ U^{-1}$ is clearly the inverse of $\Phi_U$. By a similar argument, $\Phi_{U^{-1}}$ is continuous with respect to the weak*-topology, and thus we conclude that $\Phi_U$ is a weak*-homeomorphism.
\end{proof}

\begin{corollary}\label{Linf U-inv closure}
Suppose $Y$ is a $\U$-invariant subset of $L^\infty(S)$. Then $\overline{Y}^{*}$ is $\U$-invariant.
\end{corollary}

We lastly generalize \eqref{eq:1}.

\begin{theorem}\label{switcheroo}
Let $1\leq p\leq\infty$ and let $p^\prime$ be its conjugate exponent. Then $$\int_S (f\circ U)\cdot g\,d\sigma=\int_S f\cdot (g\circ U^{-1})\,d\sigma,$$ for $f\in L^p(S)$, $g\in L^{p^\prime}(S)$, and $U\in\U$.
\end{theorem}

\section{Unitarily Invariant Spaces}\label{U inv section}

\subsection{The Weak*-Closed \texorpdfstring{$\U$}{U}-Invariant Subspaces of \texorpdfstring{$L^\infty(S)$}{L(S)}}\label{U inv spaces}

Recall that each $E_\Omega$ is $\U$-invariant, and consequently, each $\overline{E}_\Omega^{*}$ is a weak*-closed $\U$-invariant subspace of $L^\infty(S)$ by Corollary \ref{Linf U-inv closure}. The main result of this section (Theorem \ref{weak*-closed U inv}) is that the spaces $\overline{E}_\Omega^{*}$ are the \textit{only} weak*-closed $\U$-invariant subspaces of $L^\infty(S)$, and thus each weak*-closed $\U$-invariant subspace of $L^\infty(S)$ is characterized by the corresponding set $\Omega$. As was the case for Theorem \ref{main C and Lp result}, $\Omega$ is the set of all $(p,q)\in Q$ such that $\pi_{pq}$ does not annihilate the space.

\begin{theorem}\label{weak*-closed U inv}
If $Y$ is a weak*-closed $\U$-invariant subspace of $L^\infty(S)$, then $Y=\overline{E}_\Omega^{*}$ for some $\Omega\subset Q$.
\end{theorem}

The proof of Theorem \ref{weak*-closed U inv} requires Lemma \ref{main lemma} (the analogue to Lemma \ref{main C(S) lemma}), which we prove in Section \ref{proof main lemma}.

\begin{lemma}\label{main lemma}
Let $Y\subset L^\infty(S)$ be a $\U$-invariant space. Then for $g\in L^\infty(S)$, we have that $g\notin\overline{Y}^{2}$ whenever $g\notin\overline{Y}^{*}$.
\end{lemma}

We use Lemma \ref{main lemma} to establish some final observations.

\begin{remark}\label{main remark}
From Remark \ref{weak* and weak Lp closure containment} and Lemma \ref{main lemma}, we have for any $\U$-invariant subspace $Y\subset L^\infty(S)$, $$\overline{Y}^{*}=\overline{Y}^{2}\cap L^\infty(S).$$
\end{remark}

\begin{remark}\label{main remark E}
For each set $\Omega\subset Q$, the set $E_\Omega$ is a $\U$-invariant subspace of $L^\infty(S)$. Thus, Remark \ref{main remark} and Remark \ref{E Omega L2 pi pq} give a description of the sets $\overline{E}_\Omega^{*}$: $$\overline{E}_\Omega^{*}=\overline{E}_\Omega^{2}\cap L^\infty(S)=\{f\in L^\infty(S):\pi_{pq}f=0\text{ when }(p,q)\notin\Omega\}.$$
\end{remark}

\begin{proof}[Proof of Theorem \ref{weak*-closed U inv}.]
Let $Y$ be a weak*-closed $\U$-invariant subspace of $L^\infty(S)$. We observe from Remark \ref{main remark} that $$Y=\overline{Y}^{*}=\overline{Y}^{2}\cap L^\infty(S).$$

Since $Y$ is $\U$-invariant, so is $\overline{Y}^{2}$ from Corollary \ref{Lp and C U-inv closure}. As a closed $\U$-invariant subspace of $L^2(S)$, we have that $$\overline{Y}^{2}=\overline{E}_{\Omega'}^{2},$$ where $\Omega'=\{(p,q)\in Q:\pi_{pq}\overline{Y}^{2}\neq 0\}$, by Theorem \ref{main C and Lp result}.

We define $\Omega=\{(p,q)\in Q:\pi_{pq}Y\neq 0\}$. Then, Remark \ref{main remark E} yields $$\overline{E}_\Omega^{2}\cap L^\infty(S)=\overline{E}_\Omega^{*}.$$ By Theorem \ref{pi pq closure}, we have that $\Omega=\Omega^\prime$ and thus $\overline{E}_\Omega^{2}=\overline{E}_{\Omega'}^{2}$, which completes the proof.
\end{proof}

The proof of Theorem \ref{weak*-closed U inv} shows that, as in Theorem \ref{main C and Lp result}, the set $\Omega$ such that $Y=\overline{E}_\Omega^{*}$ is the set $\{(p,q):\pi_{pq}Y\neq 0\}$.

\subsection{The Weak*-Closed \texorpdfstring{$\U$}{U}-Invariant Subalgebras of \texorpdfstring{$L^\infty(S)$}{L(S)}.}\label{U inv algebras}

In this section we determine those sets $\Omega\subset Q$ whose corresponding weak*-closed $\U$-invariant spaces $\overline{E}_\Omega^{*}$ are algebras. The main result is that the subsets of $Q$ which induce closed $\U$-invariant subalgebras of $C(S)$ and those which induce weak*-closed $\U$-invariant subalgebras of $L^\infty(S)$ are exactly the same (Theorem \ref{algebra patterns}). We can then utilize results in \cite{RUA} to characterize certain weak*-closed $\U$-invariant subalgebras of $L^\infty(S)$ (Section \ref{applications}).

\begin{definition}
A set $\Omega\subset Q$ is a $C(S)$-\textbf{algebra pattern} if $\overline{E}_\Omega$ is an algebra in $C(S)$.
\end{definition}

\begin{definition}
A set $\Omega\subset Q$ is a $L^\infty(S)$-\textbf{algebra pattern} if $\overline{E}_\Omega^{*}$ is an algebra in $L^\infty(S)$.
\end{definition}

With these definitions, our result concerning the weak*-closed $\U$-invariant subalgebras of $L^\infty(S)$ can be stated in the following succinct form:

\begin{theorem}\label{algebra patterns}
A set $\Omega\subset Q$ is a $C(S)$-algebra pattern if and only if $\Omega$ is a $L^\infty(S)$-algebra pattern.
\end{theorem}

The proof of Theorem \ref{algebra patterns} requires the following necessary and sufficient condition for the set $\overline{E}_\Omega^{*}$ to be an algebra. Recall the space $H(p,q)\cdot H(r,s)$ from Definition \ref{H(p,q)H(r,s)}.

\begin{theorem}\label{nec and suff condition}
Let $\Omega\subset Q$. Then $\overline{E}_\Omega^{*}$ is an algebra if and only if $H(p,q)\cdot H(r,s)\subset\overline{E}_\Omega^{*}$ for $(p,q),(r,s)\in\Omega$.
\end{theorem}

The proof of Theorem \ref{nec and suff condition} is given in Section \ref{proof nec and suff condition}. The following lemma plays a role in the proofs of both Theorem \ref{algebra patterns} and Theorem \ref{nec and suff condition}.

\begin{lemma}\label{stitch}
Suppose $\Omega\subset Q$ and $(p,q),(r,s)\in Q$. The following are all true or all false:
\begin{itemize}
\item[(1)] $H(p,q)\cdot H(r,s)\subset E_\Omega$.
\item[(2)] $H(p,q)\cdot H(r,s)\subset \overline{E}_\Omega$.
\item[(3)] $H(p,q)\cdot H(r,s)\subset \overline{E}_\Omega^{*}$.
\end{itemize}
\end{lemma}

\begin{proof}
Observe since $E_\Omega\subset\overline{E}_\Omega\subset\overline{E}_\Omega^{*}$, we have that (1) implies (2) and (2) implies (3).

To show (3) implies (1), observe $H(p,q)\cdot H(r,s)$ is finite-dimensional $\U$-invariant subspace of $C(S)$. Thus, by Theorem \ref{main C and Lp result} there exists some set $\Omega'\subset Q$ such that $$H(p,q)\cdot H(r,s)=E_{\Omega'}.$$ Suppose $(a,b)\in\Omega'$. Assuming (3) is true, we get that $$H(a,b)\subset H(p,q)\cdot H(r,s)\subset\overline{E}_\Omega^{*}.$$ A consequence of Remark \ref{main remark} is that for any $\U$-invariant subspace $Y$ of $L^\infty(S)$, $\pi_{pq}Y=0$ if and only if $\pi_{pq}\overline{Y}^{*}=0$. As a result, we get that $H(a,b)\subset E_\Omega$. Since $E_\Omega$ must also contain the algebraic sum $E_{\Omega^\prime}$ of the sets $H(a,b)$ for $(a,b)\in\Omega'$, we have that (1) holds.
\end{proof}

\begin{proof}[Proof of Theorem \ref{algebra patterns}.]
The proof follows from applications of Lemma \ref{stitch} and Theorem \ref{nec and suff condition}.

Observe that $\overline{E}_\Omega$ is an algebra if and only if $H(p,q)\cdot H(r,s)\subset\overline{E}_\Omega$ for $(p,q),(r,s)\in\Omega$. This holds if and only if $H(p,q)\cdot H(r,s)\subset\overline{E}_\Omega^{*}$ for $(p,q),(r,s)\in\Omega$ (Lemma \ref{stitch}), which holds if and only if $\overline{E}_\Omega^{*}$ is an algebra (Theorem \ref{nec and suff condition}).

Thus, $\overline{E}_\Omega$ is an algebra if and only if $\overline{E}_\Omega^{*}$ is an algebra. Hence, $\Omega$ is a $C(S)$-algebra pattern if and only if $\Omega$ is a $L^\infty(S)$-algebra pattern.
\end{proof}

Having established Theorem \ref{algebra patterns}, we thus have that when $n\geq 3$, the weak*-closed $\U$-invariant subalgebras of $L^\infty(S)$ are characterized by those sets $\Omega\subset Q$ described in Theorem~\ref{criterion}, and that the same exceptions exist when $n=2$.

\subsection{Applications of Theorem \ref{algebra patterns}}\label{applications}

In this section we state analogues to theorems from \cite{RUA} which Theorem \ref{algebra patterns} allows us to establish. We omit the proofs, since each is a simple matter of applying the stated result of \cite{RUA} followed by Theorem \ref{algebra patterns}. Recall the sets $G(d)$, $G(\Sigma)$, and $G(p,q)$ from Definition \ref{Rudin alg patts}; by Theorem \ref{algebra patterns}, each is an $L^\infty(S)$-algebra pattern.

The following is the analogue to Theorem \ref{RUA Thm 2}.

\begin{theorem}\label{Linf G(d)}
If $\Omega$ is an $L^\infty(S)$-algebra pattern that contains some $(p,q)$ with $p>q$ and some $(r,s)$ with $r<s$, then $\Omega=G(d)$ for some $d>0$.
\end{theorem}

The following is the analogue to Theorem \ref{RUA Thm 3}.

\begin{theorem}\label{Linf G(sig)}
Suppose $\Omega$ is an $L^\infty(S)$-algebra pattern such that $p\geq q$ for every $(p,q)\in\Omega$ and $(a,a)\in\Omega$ for some $a>0$.
\begin{itemize}
    \item[(1)] If $n\geq 3$, then $\Omega=G(\Sigma)$.
    \item[(2)] If $n=2$ and if $G^*(\Sigma)$ is obtained from $G(\Sigma)$ by deleting all $(m,m)$ with $m$ odd, then $G^*(\Sigma)$ is also an algebra pattern, and $\Omega$ is either $G(\Sigma)$ or $G^*(\Sigma)$.
\end{itemize}
\end{theorem}

The following is the analogue to Theorem \ref{RUA Thm 4}.

\begin{theorem}\label{smallest Linf alg pattern}
Fix $(p,q)\in Q$ with $p>q$. Let $\Omega^\infty(p,q)$ be the smallest $L^\infty(S)$-algebra pattern containing $(p,q)$.
\begin{itemize}
	\item[(1)] If $n\geq 3$, then $\Omega^\infty(p,q)=G(p,q)$.
	\item[(2)] If $n=2$, then $\Omega^\infty(p,q)$ is obtained from $G(p,q)$ by deleting points $(mp-1,mq-1)$ for $m=2,3,4,\ldots$, and points $(2p-j,2q-j)$ for $j$ odd.
\end{itemize}
\end{theorem}

As in \cite{RUA}, each self-adjoint algebra $\overline{E}_\Omega^{*}$ is described by Theorem \ref{Linf G(d)} or Theorem \ref{Linf G(sig)} (with $\Sigma$ empty).

\subsection{The Proof of Lemma \ref{main lemma}.}\label{proof main lemma}

The proof of Lemma \ref{main lemma} requires Lemma \ref{cont map}, which is the analogue to Lemma \ref{U to X cont}.

\begin{lemma}\label{cont map}
Let $g\in L^\infty(S)$. Then the map $\varphi:\U\to L^\infty(S)$ given by $\varphi(U)=g\circ U$ is weak*-continuous.
\end{lemma}

\begin{proof}
To show that $\varphi$ is weak*-continuous, we show that each $\Lambda_h\circ\varphi$ is continuous, where $\Lambda_h$ is the map $L^\infty(S)\to\C$ given by integration against the function $h\in L^1(S)$.

Observe the map $\Lambda_h\circ\varphi:\U\to\C$ is given by $$(\Lambda_h\circ\varphi)(U)=\int_S(g\circ U)\cdot h\,d\sigma.$$ To show $\Lambda_h\circ\varphi$ is continuous, we show that it is a composition of continuous maps.

First, from Lemma \ref{U to X cont}, we have that the map $U\mapsto h\circ U$ is continuous from $\U$ into $L^1(S)$. Next, since $\U$ is a topological group, inversion in $\U$ is continuous, and hence $U\mapsto h\circ U^{-1}$ is a continuous map into $L^1(S)$.

We lastly recall that multiplication by an $L^\infty(S)$ function is a continuous map from $L^1(S)$ into $L^1(S)$. Thus, the map $U\mapsto g\cdot(h\circ U^{-1})$ is continuous. Then, since integration over $S$ is a continuous map from $L^1(S)$ into $\C$, we apply Theorem \ref{switcheroo} to get that $$U\mapsto\int_Sg\cdot (h\circ U^{-1})\,d\sigma=\int_S(g\circ U)\cdot h\,d\sigma$$ is continuous. This completes the proof.
\end{proof}

\begin{proof}[Proof of Lemma \ref{main lemma}.]
Suppose $g\in L^\infty (S)$ and $g\notin\overline{Y}^{*}$. Then there exists some weak*-continuous linear functional $\Gamma$ on $L^\infty(S)$ such that $\Gamma f=0$ for $f\in Y$, and $\Gamma g=1$. Since each weak*-continuous linear functional on $L^\infty(S)$ is induced by an element of $L^1(S)$, there exists some $h\in L^1(S)$ such that $\Gamma F=\int_S Fh\,d\sigma$ for $F\in L^\infty (S)$.

From Lemma \ref{cont map}, the map $U\mapsto\int_S (g\circ U)\cdot h\,d\sigma$ is continuous from $\U$ into $\C$. Thus, there exists a neighborhood $N$ of the identity in $\U$ such that $\text{Re}\int_S (g\circ U)\cdot h\,d\sigma>\frac{1}{2}$ for $U\in N$. We choose a continuous map $\psi:\U\to [0,\infty)$ such that $\int\psi\,dU=1$ and the support of $\psi$ is contained in $N$ (recall $dU$ denotes the Haar measure on $\U$).

We now define a map $\Lambda$ on $L^\infty (S)$ by $$\Lambda F=\int_S h(z)\int_\U \psi(U)\cdot F(Uz)\,dU\,d\sigma(z),\text{ for }F\in L^\infty(S).$$

We fix $F\in L^\infty (S)$ and $z\in S$. If we define the map $\mathscr{F}_z:\U\to\C$ by $U\mapsto F(Uz)$, we then observe that $$\int_\U|\mathscr{F}_z|^2\,dU=\int_\U |F(Uz)|^2\,dU=\int_S |F|^2\,d\sigma=\n{F}_2^2<\infty,$$ by Proposition \ref{prop147}. Then since $\psi$ is continuous, we have that $\psi\in L^2(\U)$ and $\mathscr{F}_z\in L^2(\U)$, and thus by the Schwarz Inequality, we get $$\Big|\int_\U \psi \mathscr{F}_z\,dU\Big|\leq\Big(\int_\U|\psi|^2\,dU\Big)^\frac{1}{2}\Big(\int_\U|\mathscr{F}_z|^2\,dU\Big)^\frac{1}{2}=\n{\psi}_2\cdot\n{F}_2.$$ Thus, for $F\in L^\infty (S)$, $$|\Lambda F|\leq\n{\psi}_2\cdot\n{F}_2\int_S|h|\,d\sigma=\n{\psi}_2\cdot\n{F}_2\cdot\n{h}_1.$$

The linearity of $\Lambda$ on $L^\infty(S)$ follows from linearity of the integral. Thus, $\Lambda$ defines an $L^2$-continuous linear functional on $L^\infty(S)$, and hence extends to an $L^2$-continuous linear functional $\Lambda_1$ on $L^2(S)$. By interchanging the integrals in the definition of $\Lambda$, we see that $\Lambda_1$ annihilates $Y$, since $Y$ is $\U$-invariant.

Further, we have that $$\text{Re }\Lambda_1 g=\int_\U\psi(U)\Big(\text{Re}\int_S g(Uz)\cdot h(z)\,d\sigma(z)  \Big)\,dU>\int_\U\psi(U)\cdot\frac{1}{2}\,dU=\frac{1}{2}.$$ Thus, $\Lambda_1$ is a continuous linear functional on $L^2(S)$ that annihilates $Y$ but not $g$, and hence $g\notin\overline{Y}^{2}$. This completes the proof of Lemma \ref{main lemma}.
\end{proof}

\subsection{The Proof of Theorem \ref{nec and suff condition}.}\label{proof nec and suff condition}

Observe that Theorem \ref{nec and suff condition} is not easily established due to the fact that multiplication of $L^\infty$-functions is not continuous with respect to the weak*-topology. However, this multiplication is \textit{separately} continuous (Definition \ref{def sep cont}), which is sufficient to prove Theorem \ref{nec and suff condition}.

\begin{definition}\label{def sep cont}
Let $A$ be a topological space and $m$ a map $A\times A\to A$. For each element $a\in A$, we define the maps

\begin{itemize}

    \item[] $m_a:A\to A$ given by $m_a(x)=m(a,x)$, and
    
    \item[]$m^a:A\to A$ given by $m^a(x)=m(x,a)$.
    
\end{itemize}

We say $m$ is \textbf{separately continuous} if the maps $m_a$ and $m^a$ are continuous for $a\in A$. Further, if $B\subset A$, we say $B$ is \textbf{invariant} under $m$ if $m(B\times B)\subset B$.
\end{definition}

\begin{proposition}\label{mf weak* cont}
Let $f\in L^\infty(S)$. The left multiplication operator $m_f:L^\infty (S)\to L^\infty(S)$ defined by $m_f(h)=fh$ is weak*-continuous.
\end{proposition}

\begin{proof}
Let $\Lambda$ be a weak*-continuous linear functional on $L^\infty(S)$ and $g\in L^1(S)$ be such that $\Lambda h=\int_S hg\,d\sigma$ for $h\in L^\infty(S)$. We observe then that $$(\Lambda\circ m_f)(h)=\Lambda(fh)=\int_S (fh)g\,d\sigma=\int_S h(fg)\,d\sigma.$$ Since $fg\in L^1(S)$, we see that $\Lambda\circ m_f$ is a weak*-continuous linear functional on $L^\infty(S)$. We conclude that $m_f$ is weak*-continuous on $L^\infty(S)$.
\end{proof}

We can similarly define a right multiplication operator $m^g:L^\infty(S)\to L^\infty(S)$ given by $m^g(h)=hg$. Since pointwise multiplication for functions is commutative, we have $m^g=m_g$. In particular, the right multiplication operators are also weak*-continuous.

\begin{remark}\label{m sep cont}
By Proposition \ref{mf weak* cont}, the multiplication map $m:L^\infty(S) \times L^\infty(S)\to L^\infty(S)$ given by $m(f,g)=fg$ is separately continuous with respect to the weak*-topology.
\end{remark}

\begin{lemma}\label{sep cont}
Let $A$ be a topological space and $m$ a map $A\times A\to A$. Suppose $B\subset A$ with $B$ invariant under $m$. If $m$ is separately continuous, then $\overline{B}$ is invariant under $m$.
\end{lemma}

\begin{proof}
We begin by proving the following claim: If $f\in\overline{B}$, then $m_f(B)\subset\overline{B}$.

Let $g\in B$. Since $B$ is invariant under $m$, we get that $m^g(B)\subset B$. Observe since $m$ is separately continuous, the map $m^g$ is continuous, and thus we have $$m_f(g)=m^g(f)\in m^g\big(\overline{B}\big)\subset\overline{m^g(B)}\subset\overline{B}.$$ Since $g\in B$ was arbitrary, this proves the claim.

We are now ready to prove the lemma.

Let $f,g\in\overline{B}$. Then, we have that $$m(f,g)=m_f(g)\in m_f(\overline{B})\subset\overline{m_f(B)},$$ since $m_f$ is continuous. However, by the above claim, $m_f(B)\subset\overline{B}$, and thus $m(f,g)\in\overline{B}$.

Since we chose $f,g$ arbitrarily from $\overline{B}$, we conclude that $\overline{B}$ is invariant under $m$.
\end{proof}

\begin{proof}[Proof of Theorem \ref{nec and suff condition}.]
The forward direction is straightforward. We suppose $\overline{E}_\Omega^{*}$ is an algebra and fix pairs $(p,q),(r,s)\in\Omega$. Let $h\in H(p,q)\cdot H(r,s)$. Then, we have that $$h=\sum_{k=1}^nf_kg_k,$$ where $f_k\in H(p,q)$ and $g_k\in H(r,s)$ for all $k$. Since $H(a,b)\subset E_\Omega\subset\overline{E}_\Omega^{*}$ for $(a,b)\in\Omega$, we have that $f_k,g_k\in\overline{E}_\Omega^{*}$ for all $k$. Since $\overline{E}_\Omega^{*}$ is an algebra, we have that $f_kg_k\in\overline{E}_\Omega^{*}$ for all $k$, and further $h\in\overline{E}_\Omega^{*}$. Thus, $H(p,q)\cdot H(r,s)\subset\overline{E}_\Omega^{*}$, which completes the forward direction.

We now show the backward direction. Suppose $H(p,q)\cdot H(r,s)\subset\overline{E}_\Omega^{*}$ for $(p,q),(r,s)\in\Omega$. To show that $\overline{E}_\Omega^{*}$ is an algebra, we must verify that $\overline{E}_\Omega^{*}$ is invariant under the multiplication map $m:L^\infty(S)\times L^\infty(S)\to L^\infty(S)$ given by $m(f,g)=fg$ for $f,g\in L^\infty(S)$. By Remark~\ref{m sep cont}, this multiplication map is separately continuous. Thus, by showing that $E_\Omega$ is invariant under $m$, we can get that $\overline{E}_\Omega^{*}$ is invariant under $m$ by Lemma \ref{sep cont}.

Let $f,g\in E_\Omega$. Then,

\begin{itemize}

    \item[] $f=\sum_{i=1}^nf_i$, where $f_i\in H(p_i,q_i)$ and $(p_i,q_i)\in\Omega$, for $1\leq i\leq n$, for some $n$, and
    
    \item[] $g=\sum_{j=1}^mg_j$, where $g_j\in H(r_j,s_j)$ and $(r_j,s_j)\in\Omega$, for $1\leq j\leq m$, for some $m$.

\end{itemize}

Further, we have that $$fg=\sum_{i=1}^n\sum_{j=1}^mf_ig_j,\text{ and }f_ig_j\in H(p_i,q_i)\cdot H(r_j,s_j)\text{ for all }i,j.$$ Since $H(p_i,q_i)\cdot H(r_j,s_j)\subset\overline{E}_\Omega^{*}$ for all $i,j$, we have that $H(p_i,q_i)\cdot H(r_j,s_j)\subset E_\Omega$ for all $i,j$ by Lemma \ref{stitch}. Hence, $f_ig_j\in E_\Omega$ for all $i,j$. As a subspace, $E_\Omega$ is in particular closed under finite sums, and thus $fg\in E_\Omega$. Since $f$ and $g$ were chosen arbitrarily in $E_\Omega$, we conclude that $E_\Omega$ is invariant under $m$.

As stated before, this allows us to apply Lemma \ref{sep cont}. Thus, $\overline{E}_\Omega^{*}$ is invariant under $m$, and hence $\overline{E}_\Omega^{*}$ is an algebra.
\end{proof}

\section{M\"obius Invariant Spaces}\label{M inv section}

\subsection{The Weak*-Closed \texorpdfstring{$\M$}{M}-Invariant Subspaces of \texorpdfstring{$L^\infty(S)$}{L(S)}}\label{M inv spaces}

In this section we determine the weak*-closed $\M$-invariant subspaces of $L^\infty(S)$ (Theorem \ref{M-inv subspaces of Linf}). We begin by noting that every $\M$-invariant subset of $L^\infty(S)$ is also $\U$-invariant, since every bijective holomorphic map from $B$ to $B$ is a unitary transformation. Thus by Theorem~\ref{weak*-closed U inv}, determining the weak*-closed $\M$-invariant subspaces of $L^\infty(S)$ only requires finding those sets $\Omega\subset Q$ such that $\overline{E}_\Omega^{*}$ is $\M$-invariant. Similarly to Theorem B of \cite{NR}, there are only six such spaces.

\begin{theorem}\label{M-inv subspaces of Linf}
The following are the weak*-closed $\M$-invariant subspaces of $L^\infty(S)$:
\begin{itemize}
\item[1.] the null space $\{0\}$,
\item[2.] the space of constant functions $\C$,
\item[3.] the space $H^\infty(S)$,
\item[4.] the space $\text{Conj}(H^\infty(S))$,
\item[5.] the weak*-closure of the space $H^\infty(S)+\text{Conj}(H^\infty(S))$, and
\item[6.] the space $L^\infty(S)$.
\end{itemize}
\end{theorem}

The proof of this result uses the following lemmas. The first (Lemma \ref{main M lemma}) is the analogue to Lemma \ref{M C(S) lemma}, which it also requires in its proof. The use of the second, Lemma \ref{Omega 3}, is clear from the statement.

\begin{lemma}\label{main M lemma}
Let $Y$ be a weak*-closed $\M$-invariant subspace of $L^\infty(S)$, $p\geq 1$ an integer, and $H(p,q)\subset Y$ for some $q\geq 0$. Then $H(p-1,q)\subset Y$ and $H(p+1,q)\subset Y$.
\end{lemma}

\begin{proof}
Since $Y$ is $\M$-invariant, $Y$ is in particular $\U$-invariant, and thus there exists some $\Omega\subset Q$ such that $Y=\overline{E}_\Omega^{*}$ by Theorem \ref{weak*-closed U inv}. Consequently, we get $H(p,q)\subset E_\Omega$.

We now define $Y_C=Y\cap C(S)$. Since $Y$ is weak*-closed in $L^\infty(S)$, we have that $Y$ is norm-closed in $L^\infty(S)$, and hence $Y_C$ is closed in $C(S)$.

Further, $Y_C$ is $\M$-invariant as the intersection of $\M$-invariant spaces. We also observe that $E_\Omega\subset Y_C$ and hence $H(p,q)\subset Y_C$.

Thus by Lemma \ref{M C(S) lemma}, we have $H(p-1,q)\subset Y_C$ and $H(p+1,q)\subset Y_C$, and the desired result follows accordingly.
\end{proof}

\begin{lemma}\label{Omega 3}
Let $\Omega$ be the set $\{(p,q):q=0\}$. Then the space $\overline{E}_{\Omega}^{*}$ is the space $H^\infty(S)$.
\end{lemma}

The proof of Lemma \ref{Omega 3} uses some results concerning the Poisson kernel and the Poisson integrals of $L^1$-functions. As such, we save the proof of Lemma \ref{Omega 3} for Section \ref{proof Omega 3}.

\begin{proof}[Proof of Theorem \ref{M-inv subspaces of Linf}.]
Since the roles of $p$ and $q$ can be switched in Lemma \ref{M C(S) lemma}, the roles of $p$ and $q$ in Lemma \ref{main M lemma} can also be switched. Thus, if $Y$ is a weak*-closed $\M$-invariant subspace of $L^\infty(S)$ and $\Omega$ is the subset of $Q$ such that $\overline{E}_\Omega^{*}=Y$, repeated application of Lemma \ref{main M lemma} yields only the following six possible options for the set $\Omega$:

\begin{itemize}
\item[(1)] $\Omega=\emptyset$.
\item[(2)] $\Omega=\{(0,0)\}$.
\item[(3)] $\Omega=\{(p,q):q=0\}$.
\item[(4)] $\Omega=\{(p,q):p=0\}$.
\item[(5)] $\Omega=\{(p,q):\text{Either }p=0\text{ or }q=0\}$.
\item[(6)] $\Omega=Q$.
\end{itemize}

These are the same six possible options found in \cite{NR}. We now determine the spaces they induce. This process is easiest for the sets in (1), (2), (6), and a special case for (5).

When $\Omega=\emptyset$, we have that the corresponding weak*-closed $\M$-invariant subspace is the algebra consisting only of the zero function. To show this, observe by Remark \ref{main remark E}, $$\overline{E}_{\Omega}^{*}=\{f\in L^\infty(S):\pi_{pq}f=0\text{ for all }(p,q)\}.$$ Clearly, $0\in\overline{E}_{\Omega}^{*}$, and further, if $g\in\overline{E}_{\Omega}^{*}$, we have $g=\sum\pi_{pq}g=0$.

When $\Omega=\{(0,0)\}$, we have that the corresponding weak*-closed $\M$-invariant subspace is the algebra of constant functions. To show this, observe by Remark \ref{main remark E}, $$\overline{E}_{\Omega}^{*}=\{f\in L^\infty(S):\pi_{pq}f=0\text{ for all }(p,q)\neq(0,0)\}.$$ Clearly, $\C\subset\overline{E}_{\Omega}^{*}$, and further, if $g\in\overline{E}_{\Omega}^{*}$, we have $g=\sum\pi_{pq}g=\pi_{00}g\in\C$.

For $\Omega=Q$, we have that $\overline{E}_{\Omega}^{*}=L^\infty(S)$. This follows from Remark \ref{main remark} and Theorem \ref{L2 Hpq structure}, which states that $\overline{E}_Q^{2}=L^2(S)$.

As a special case for the set in (5), when $n=1$ we get that $\overline{E}_{\Omega}^{*}=L^\infty(S)$. This is due to the fact that either $p=0$ or $q=0$ for each space $H(p,q)$ when $n=1$. Thus, the set $\Omega$ consists of all possible pairs $(p,q)$; that is, $\Omega=Q$. Since $\overline{E}_Q^{*}=L^\infty(S)$, we get that $\overline{E}_{\Omega}^{*}=L^\infty(S)$.

Determining the spaces $\overline{E}_{\Omega}^{*}$ for the remaining options for $\Omega$ ($n>1$ for the fifth option) is not as easy. Fortunately, the descriptions of these spaces from Remark \ref{main remark E} show that determining $\overline{E}_{\Omega}^{*}$ for the set in (3) will suffice to get the remaining two spaces.

By Lemma \ref{Omega 3}, we have that $\overline{E}_{\Omega}^{*}=H^\infty(S)$ when $\Omega=\{(p,q):q=0\}$. Observe then for any nonnegative integer $p$, $$H(0,p)=\text{Conj}(H(p,0)).$$ Consequently, we get $\overline{E}_{\Omega}^{*}=\text{Conj}(H^\infty(S))$ when $\Omega=\{(p,q):p=0\}$.

When $\Omega=\{(p,q):\text{Either }p=0\text{ or }q=0\}$, we similarly have that $$\overline{E}_{\Omega}^{*}=\{f\in L^\infty(S):\pi_{pq}f=0\text{ when }p>0\text{ and }q>0\},$$ which we note is the weak*-closure of the algebraic sum $H^\infty(S)+\text{Conj}(H^\infty(S))$ by the previous two cases.
\end{proof}

\subsection{The Weak*-Closed \texorpdfstring{$\M$}{M}-Invariant Subalgebras of \texorpdfstring{$L^\infty(S)$}{L(S)}}\label{M inv algebras}

In this section we determine the weak*-closed $\M$-invariant subalgebras of $L^\infty(S)$ (Theorem \ref{M-inv subalgebras of Linf}). Our process is fairly straightforward since we only have the six spaces of Theorem \ref{M-inv subspaces of Linf} to consider.

\begin{theorem}\label{M-inv subalgebras of Linf}
The following are the weak*-closed $\M$-invariant subalgebras of $L^\infty(S)$:
\begin{itemize}
\item[1.] the null space $\{0\}$,
\item[2.] the space of constant functions $\C$,
\item[3.] the space $H^\infty(S)$,
\item[4.] the space $\text{Conj}(H^\infty(S))$,
\item[5.] the weak*-closure of the space $H^\infty(S)+\text{Conj}(H^\infty(S))$ if and only if $n=1$, and
\item[6.] the space $L^\infty(S)$.
\end{itemize}
\end{theorem}

\begin{proof}
The spaces listed in the above statement are the weak*-closed $\M$-invariant subspaces of $L^\infty(S)$ by Theorem \ref{M-inv subspaces of Linf}. Further, all except the weak*-closure of $H^\infty(S)+\text{Conj}(H^\infty(S))$ are clearly algebras for all $n$. Recall when $n=1$, we have that the weak*-closure of $H^\infty(S)+\text{Conj}(H^\infty(S))$ is $L^\infty(S)$, which is clearly an algebra. We thus need only show that for $n>1$, the weak*-closure of $H^\infty(S)+\text{Conj}(H^\infty(S))$ is not an algebra.

We fix $\Omega=\{(p,q):\text{Either }p=0\text{ or }q=0\}$ and consider the coordinate functions $f(z)=z_1$ and $g(z)=\bar{z}_n$. Clearly, both $f$ and $g$ are elements of the weak*-closure of $H^\infty(S)+\text{Conj}(H^\infty(S))$, since $f\in H(1,0)\subset E_{\Omega}$ and $g\in H(0,1)\subset E_{\Omega}$. However, their product, given by $(fg)(z)=z_1\bar{z}_n$, is an element of $H(1,1)$. Hence $fg\notin\overline{E}_{\Omega}^{*}$ and thus the weak*-closure of $H^\infty(S)+\text{Conj}(H^\infty(S))$ is not an algebra when $n>1$.
\end{proof}

\subsection{The Proof of Lemma \ref{Omega 3}}\label{proof Omega 3}

The proof of Lemma \ref{Omega 3} uses the following lemma in addition to standard results concerning the Poisson integrals of $L^1$-functions. These results can be found in many texts, one of which is \cite{HFT}. 

\begin{lemma}\label{pi pq holomorphic}
Let $f$ be holomorphic on $B$ and $0<r<1$. If we define $f_r(\zeta)=f(r\zeta)$ for $\zeta\in S$, then $\pi_{pq}f_r=0$ whenever $q>0$.
\end{lemma}

\begin{proof}
Since $f$ is holomorphic on $B$, $f$ has a power series representation $\sum f_\alpha z^\alpha$ on $B$ which converges absolutely and uniformly to $f$ on compacta of $B$ (\cite{RFT}).

Observe then that $f_r\in C(S)$ and $\sum f_\alpha r^{|\alpha|}\zeta^\alpha$ converges uniformly to $f_r$ on $S$. In particular, $\sum f_\alpha r^{|\alpha|}\zeta^\alpha$ converges to $f_r$ in $L^2(S)$, and hence is the unique $L^2$ expansion of $f_r$.

Clearly then, $\pi_{pq}f_r=0$ whenever $q>0$.
\end{proof}

The next two theorems follow from standard results of Poisson integrals.

\begin{theorem}\label{P[f] bdd}
If $f\in L^\infty(S)$, then $P[f]$ is bounded on $B$.
\end{theorem}

\begin{theorem}\label{P[f] holom}
If $f$ is holomorphic on $B$ and continuous on $\overline{B}$, then $f(z)=P[f](z)$ for $z\in B$. Consequently, $P[f]$ is holomorphic on $B$. 
\end{theorem}

\begin{proof}[Proof of Lemma \ref{Omega 3}.]
Let $\Omega=\{(p,q):q=0\}$. The proof is split into two parts, each part verifying a set inclusion. In the first part, we show that $H^\infty(S)\subset\overline{E}_\Omega^{*}$.

Suppose that $f\in H^\infty(S)$. To show that $f\in\overline{E}_\Omega^{*}$, we verify that $\pi_{pq}f=0$ whenever $q>0$. Let $g\in H^\infty(B)$ be the function such that $$\lim_{r\to 1}g(r\zeta)=f(\zeta)$$ for almost all $\zeta\in S$. For each positive integer $m$, we define the function $g_m$ on $S$ by $$g_m(\zeta)=g\Big(\frac{m-1}{m}\zeta\Big)$$ for $\zeta\in S$. Then, $$\lim_{m\to\infty}g_m(\zeta)=f(\zeta)$$ for almost all $\zeta\in S$.

We now fix $(p,q)\in Q$ and $z\in S$. Recall from Theorem \ref{L2 Hpq structure}(a), there exists a unique element $K_z\in H(p,q)$ such that $$(\pi_{pq}h)(z)=\int_S h\overline{K}_z\,d\sigma,$$ for $h\in L^2(S)$. Clearly, then, for $z\in S$, we have that $$\lim_{m\to\infty}g_m(\zeta)\overline{K}_z(\zeta)=f(\zeta)\overline{K}_z(\zeta)$$ for almost all $\zeta\in S$.

Each $K_z\in C(S)$ and hence has uniform norm $\n{K_z}$. Further, since $g\in H^\infty(B)$, $g$ has uniform norm $\n{g}$. Thus, for each $z\in S$, $$|g_m(\zeta)\overline{K}_z(\zeta)|\leq\n{g}\cdot\n{K_z},$$ for $\zeta\in S$. Thus, by the Lebesgue Dominated Convergence Theorem, $$\lim_{m\to\infty}\int_S g_m(\zeta)\overline{K}_z(\zeta)\,d\sigma(\zeta)=\int_S f(\zeta)\overline{K}_z(\zeta)\,d\sigma(\zeta).$$ Hence, $(\pi_{pq}g_m)(z)\to (\pi_{pq}f)(z)$ as $m\to\infty$, for $z\in S$.

From Lemma \ref{pi pq holomorphic}, we have that $\pi_{pq}g_m=0$ whenever $q>0$. Hence, $\pi_{pq}f=0$ whenever $q>0$, and thus $f\in\overline{E}_\Omega^{*}$. This completes the first part of the proof.

In the second part of the proof, we show that $\overline{E}_\Omega^{*}\subset H^\infty(S)$. Suppose $f\in\overline{E}_\Omega^{*}$. To show that $f\in H^\infty(S)$, we must find some function $g\in H^\infty(B)$ such that $$\lim_{r\to 1}g(r\zeta)=f(\zeta)$$ for almost all $\zeta\in S$.

We let $P[f]$ denote the Poisson integral of $f$. By Theorem \ref{P[f] bdd}, $P[f]$ is bounded on $B$ and further, $P[f]$ has radial limits $f(\zeta)$ at almost every $\zeta\in S$. We lastly show that $P[f]$ is holomorphic on $B$ as the uniform limit on compacta of a sequence of holomorphic functions.

Recall from Remark \ref{f_pq=pi_pq f}, the sum $\sum\pi_{pq}f$ converges absolutely and unconditionally to $f$ in the $L^2$-norm. Let $\pi_p$ denote the map $\pi_{p0}$ for all $p$. the previous sum then becomes $$\sum_{p=0}^\infty\pi_pf,$$ since $\pi_{pq}f=0$ for $q>0$. We then have $$f=\sum_{p=0}^\infty\pi_pf,$$ with convergence in the $L^2$-norm.

Each $\pi_pf$ is a holomorphic polynomial as an element of $H(p,0)$, and so further is each partial sum $s_k$, defined by $$s_k=\sum_{p=0}^k\pi_pf.$$ Observe $s_k$ converges to $f$ in the $L^2$-norm, and hence in the $L^1$-norm. Since each $s_k$ is holomorphic on $B$ and continuous on $\overline{B}$, the corresponding Poisson integral $P[s_k]$ is holomorphic on $B$ by Theorem \ref{P[f] holom}.

Let $K$ be a compact subset of $B$. Since the Poisson kernel is bounded on compact subsets of $B\times S$, there exists some $M_K>0$ such that $$P(z,\zeta)\leq M_K\text{ for }(z,\zeta)\in K\times S.$$

Suppose now that $z\in K$. We thus have $$\Big|P[s_k](z)-P[f](z)\Big|\leq\int_S|s_k(\zeta)-f(\zeta)|P(z,\zeta)\,d\sigma(\zeta)\leq M_K\int_S|s_k(\zeta)-f(\zeta)|\,d\sigma(\zeta).$$ The integral farthest to the right above goes to 0, since $s_k\to f$ in the $L^1$-norm, and hence $$P[s_k]\to P[f]$$ uniformly on $K$. We conclude that $P[s_k]$ converges to $P[f]$ uniformly on compacta of $B$, and thus, $P[f]$ is holomorphic on $B$.

Since $f$ is the radial limit almost everywhere on $S$ of a bounded holomorphic function on $B$, we conclude that $f\in H^\infty(S)$, which completes the proof.
\end{proof}

\bibliographystyle{unsrt}
\bibliography{bibliography}

\end{document}